\documentclass[a4paper, 12pt]{article}
\usepackage{hyperref}
\usepackage{amsmath}%
\usepackage{amsthm}
\usepackage{amsfonts}%
\usepackage{amssymb}%
\usepackage{enumitem}
\usepackage{graphicx}
\usepackage{fullpage}
\usepackage{color}
\usepackage[utf8]{inputenc}
\usepackage{lmodern}
\usepackage[giveninits=true, url=false, isbn=false]{biblatex}
\AtBeginBibliography{\small}

\usepackage{soul}
\usepackage{cancel}

\makeatletter
\newcommand{\neutralize}[1]{\expandafter\let\csname c@#1\endcsname\count@}
\makeatother

\newcommand{\Nb}{\mathbb{N}}
\newcommand{\Nbz}{\mathbb{N}_0}
\newcommand{\Cb}{\mathbb{C}}
\newcommand{\Zb}{\mathbb{Z}}
\newcommand{\Qb}{\mathbb{Q}}
\newcommand{\Qbp}{\Qb_{\geq0}}
\newcommand{\Rb}{\mathbb{R}}
\newcommand{\Sb}{\mathbb{S}}

\newcommand{\Div}{\mathrm{div}}
\newcommand{\grad}{\mathrm{grad}}

\newcommand{\tr}{\mathrm{tr}}

\newcommand{\me}{\mathrm{e}}
\newcommand{\mi}{\mathrm{i}}

\newtheorem{theorem}{Theorem}[section]

\newtheorem{conjecture}[theorem]{Conjecture}
\newtheorem{corollary}[theorem]{Corollary}

\newtheorem{definition}[theorem]{Definition}

\newtheorem{keyid}[theorem]{Key Identity}
\newtheorem{lemma}[theorem]{Lemma}

\newtheorem{proposition}[theorem]{Proposition}
\newtheorem*{remark}{Remark}

\newtheorem{thm}{Conjecture}

\title{On multiplicity bounds for eigenvalues of the  clamped round plate}
\author{Dan Mangoubi and Daniel Rosenblatt}
\date{}

\begin{document}
\maketitle
\begin{abstract}
We ask whether the only multiplicities in the spectrum of the clamped round plate are trivial, i.e., whether all existing multiplicities are  due to the isometries of the sphere, or, equivalently,  whether any eigenfunction is separated. We prove
that any eigenfunction can be expressed as a sum of at most two separated ones, by showing that otherwise the corresponding  eigenvalue is algebraic,  contradicting  the Siegel-Shidlovskii theory. In two dimensions
  it follows that no eigenvalue is of multiplicity greater than four. The proof exploits a linear recursion  of order two for cross-product Bessel functions with  coefficients which are not even algebraic functions, though they do satisfy a non-linear algebraic recursion. 
\end{abstract}
\section{Introduction and background}
\subsection{The vibrating clamped round plate}
    In this paper we are concerned with
the vibrating clamped round  plate~\cite[ch. V\S6]{courant-hilbert}, that is, the following fourth order eigenvalue problem in the unit ball~$B\subset\Rb^d$.
    \begin{equation*}
    \text{(VP)}\qquad\left\{\begin{array}{rcl}
         \Delta^2u&=&\lambda u \quad \text{ in } B,\\ 
         u&=&0  \qquad \text{on } \partial B,\\ 
         \partial_n u&=&0  \qquad \text{on } \partial B.
                \end{array}\right.
\end{equation*} 
where $\Delta=\Div\circ\grad$ is the  Laplacian. A basis of eigenfunctions is given in spherical coordinates by the family
\begin{equation}
\label{eqn:eigen-basis}
    u_{l, k, j}(r,\phi)=\bigl(I^{sp}_l(w^{sp}_{l,k})J^{sp}_l(w^{sp}_{l,k}r)-J^{sp}_l(w^{sp}_{l,k})I^{sp}_l(w^{sp}_{l,k}r)\bigr)Y_{l, j}(\phi)
\end{equation}
where 
\[ J^{sp}_l(\rho)=\rho^{1-\frac{d}{2}}J_{l+\frac{d}{2}-1}(\rho),\quad I^{sp}_l(\rho)=\rho^{1-\frac{d}{2}}I_{l+\frac{d}{2}-1}(\rho)\]
denote the spherical Bessel function of the first kind and the modified spherical Bessel function of the first kind respectively, and where 
$(Y_{l, j})_{j=1}^{N_l}$ with $N_l=\binom{l+d-1}{d-1}-\binom{l+d-3}{d-1}$  form an orthonormal basis for  spherical harmonics of degree~$l$ on the sphere~$\Sb^{d-1}$. The value $w^{sp}_{l, k}$ is the~$k$-th positive root of the Wronskian

\begin{equation*}
    W^{sp}_{l} :=(I^{sp}_l)'J^{sp}_l-I^{sp}_l (J^{sp}_l)'=I^{sp}_{l+1}J^{sp}_{l}+I^{sp}_{l} J^{sp}_{l+1}=I^{sp}_{l-1}J^{sp}_l-I^{sp}_l J^{sp}_{l-1}\ .
\end{equation*}

The eigenvalue corresponding to an eigenfunction of the form~\eqref{eqn:eigen-basis} is~$\lambda=(w^{sp}_{l, k})^4$. The multiplicities in the spectrum which we call ``trivial'' are those due to the  multiplicities appearing in the spectrum of  the sphere~$\Sb^{d-1}$.
We are interested to know whether non-trivial multiplicities occur in the spectrum, i.e., whether an eigenvalue~$\lambda=w^{sp}_{l, k}$ can have multiplicity bigger than~$N_l$, or, equivalently,  whether any eigenfunction~$u$ is of the separated form \begin{equation}
\label{eqn:separated}
    u(r, \phi)=R(r)Y(\phi)
\end{equation}
where~$Y$ is a spherical harmonic.  The analogous problem for the vibrating round \emph{membrane} with Dirichlet boundary conditions was solved by Siegel in~\cite{siegel} (see~\cite{porter} for the odd dimensional case; cf.~\cite[\S 15.28]{watson},~\cite{siegel-book} and~\cite[p.~217]{shid-book}), proving that any eigenfunction is separated. To that end, he proved the deep fact that all non-zero roots of Bessel functions~$J_m$ where $m\in\Qb$ and $2m$ is not an odd integer are transcendental.
On the other hand, any eigenvalue of non-trivial multiplicity must be algebraic due to the algebraic recursion formula satisfied by Bessel functions.
In this paper we continue and further improve the bounds which were obtained in~\cite{lvov-man}.
  Ruling out  non-trivial multiplicities in the case of the free vibrating round membrane problem was achieved in~\cite{ashu2013} and~\cite{helf-sund}.  Bounds on the multiplicities of eigenvalues for annular domains with Dirichlet boundary conditions were studied in~\cite{wang-annuli}.

\subsection{Common roots of cross-product  Bessel functions}
One can write 
\[W^{sp}_l(\rho)=\rho^{2-d}W_{l+\frac{d}{2}-1}(\rho)\]
where 
\begin{equation}
 \label{eqn:W-def}
W_m:=I_m' J_m-I_m J_m'=I_{m+1}J_m+I_m J_{m+1}=I_{m-1}J_m -I_m J_{m-1}
\end{equation}
In particluar, $w^{sp}_{l, k}=w_{l+d/2-1, k}$ where $w_{m, k}$ denotes the $k$-th positive root of $W_m$. 
   Hence, the problem treated in this paper amounts to the question whether there exist $x_0>0$ and $m_1, m_2\in \frac{1}{2}\Nb_{0}$ differing by a non-zero integer such that  $W_{m_1}(x_0)=W_{m_2}(x_0)=0$, as it would imply that $\lambda=x_0^4$ is an eigenvalue of non-trivial multiplicity (at least~$N_{l_1}+N_{l_2}$ where $m_j=l_j+d/2-1$). This seems to be a difficult open problem.  In~\cite{lvov-man} it was shown that there do not exist $x_0>0$ and $m_1, m_2, m_3, m_4\in \Nb_0$ pairwise distinct  for which $W_{m_1}(x_0)=W_{m_2}(x_0)=W_{m_3}(x_0)=W_{m_4}(x_0)=0$. In two dimensions this implied a uniform bound, namely, six,  on the multiplicity. The proof was based on a~\emph{fourth} order recursion formula for the sequence $W_m$ with rational functions as coefficients.
One of the main goals of this paper is to eliminate the possibility that three functions~$W_m$ vanish simultaneously. More precisely,  we prove
\begin{theorem}
    \label{thm:mult3}
    There do not exist  $x_0>0$ and  $m_1,m_2,m_3\in \Qbp$ with $|m_i-m_j|\in\Nb$ for all $1\leq i<j\leq 3$   such that $W_{m_1}(x_0)=W_{m_2}(x_0)=W_{m_3}(x_0)=0$. Equivalently, any eigenfunction of the clamped disk is a sum of at most two ones of the separated form~\eqref{eqn:separated}. 
\end{theorem}

One corollary is
\begin{corollary}
\label{cor:dim2}
  Any eigenvalue  of the two dimensional clamped round  plate is of multiplicity at most four. In particular, the multiplicities are uniformly bounded.
\end{corollary}

\subsection{Schanuel's conjecture and the trivial multiplicity conjecture}
A natural conjecture is the following one.
\begin{conjecture}[Trivial multiplicity]
\label{conj:trivial-mult}
    There do not exist $x_0>0$ and $m_1, m_2\in \frac{1}{2}\Nb_0$ with $|m_1-m_2|\in\Nb$
    such that $W_{m_1}(x_0)=W_{m_2}(x_0)=0$. Equivalently, there are no non-trivial multiplicities in the spectrum of the  clamped round plate, or, alternatively, any eigenfunction of the clamped round plate is separated.
\end{conjecture}

In a similar vein to the work of Porter~\cite{porter} we explain in \S\ref{sec:schanuel} that  in odd dimensions 
a special case of the classical Schanuel conjecture of Transcendental Number Theory implies Conjecture~\ref{conj:trivial-mult}. For even dimensions, a Schanuel-type conjecture for Bessel functions which would imply Conjecture~\ref{conj:trivial-mult} is formulated.
\subsection{Idea of proof of main Theorem~\ref{thm:mult3}}
To prove Theorem~\ref{thm:mult3} we show that if there exist~$x_0>0$ and $m_1, m_2, m_3\in \Qbp$ pairwise differing by a non-zero integer, with $W_{m_1}(x_0)=W_{m_2}(x_0)=W_{m_3}(x_0)=0$, then $x_0$ must be algebraic. However, an immediate application of the Siegel-Shidlovskii theory shows that any positive root of the equation $W_m(x_0)=0$ is transcendental.
The main new idea in our proof with respect to~\cite{lvov-man} is that it is possible to exploit a~\emph{second} order linear recursion formula for the sequence~$W_m$ which has \emph{non-algebraic} coefficients. We make use of the fact that these coefficients satisfy an  \emph{algebraic non-linear} recursion formula of degree two. At a first step we show that each joint root~$x_0$ of  $W_{m}$ and $W_{m'}$ leads to an equation  of the form
\begin{equation}
\label{eqn:poly-f-eqn}
    P_{m, m'}(x_0, F_m(x_0))=0
\end{equation} where $P_{m, m'}(x, y)$ is a polynomial of degree two with respect to~$y$, and~$F_m$ is a quotient of successive modified Bessel functions (see Def.~\ref{def:besselq}). At a second step we prove that it is possible to  eliminate $F_m$ from a system of~\emph{any} two such equations, leading to a non-trivial polynomial equation for~$x_0$.

\subsection{Acknowledgments}
We are very grateful to Gal Binyamini, suggesting the possibility of eliminating~$F_m$ from our equations~\eqref{eqn:poly-f-eqn}.
We thank Or Kuperman for helpful discussions regarding the second order recurrence relation satisfied by the sequence~$W_m$. We are very grateful to the anonymous referee for helpful questions and comments which led us to better describe the background  to this work, add the higher dimensional case and the relation of the optimal bounds to Schanuel's conjecture.
 This paper
is part of D.\ R.'s research towards a PhD.\ degree, conducted
at the Hebrew University of Jerusalem.

\section{Bessel functions and their quotients}
Let $m\in\Qb$. The Bessel function $J_m$ is
defined  by the series
$$J_m(x)=\Bigl(\frac{x}{2}\Bigr)^{m}\sum_{k=0}^{\infty} \frac{(-1)^k}{k!\Gamma(m+k+1)}\Bigl(\frac{x}{2}\Bigr)^{2k}\ .$$ For the purpose of this paper we  consider the above series as a holomorphic function in the domain~$\Cb\setminus (-\infty, 0]$ which is real on the positive real axis.
Similarly, the modified Bessel function $I_m$ is given by
$$I_m(x)=\Bigl(\frac{x}{2}\Bigr)^{m}\sum_{k=0}^{\infty} \frac{1}{k!\Gamma(m+k+1)}\Bigl(\frac{x}{2}\Bigr)^{2k}\ .$$
Observe that $I_m(\mi x)=\mi^m J_m(x)$. If $m$ is a negative integer we ignore the terms for which $m+k+1$ is a pole of the~$\Gamma$-function. In view of this, $I_{-m}=I_m$ when~$m\in\Zb$.
\begin{lemma}
    \label{lem:I-positive}
    If $m\in\Zb$ or~$m\in\Qbp$ and $x>0$, then $I_m(x)>0$.
\end{lemma}
\begin{proof}
    For $m\geq 0$ all the terms of the power series are positive for $x>0$.
    If $m<0$ is an integer, we have $I_{-m}=I_m$.
\end{proof}
We record the following formulas for a few special cases which follow directly from the definitions. 
\begin{lemma}
\label{lem:JIonehalf}
\[J_{-1/2}(x)=\sqrt{\frac{2}{\pi x}}\cos x,\quad J_{1/2}(x)=\sqrt{\frac{2}{\pi x}}\sin x \]
\[ I_{-1/2}(x)= \sqrt{\frac{2}{\pi x}}\cosh(x),\quad I_{-1/2}(x)=\sqrt{\frac{2}{\pi x}}\sinh(x) \]
\end{lemma}
\begin{proposition}[{\cite[\S2.12, \S3.71]{watson}}]
\label{prop:besselrec}
Let $m\in\Qb$. The following recursions are satisfied.
\begin{equation*}
    \begin{split}
        J_{m+1}(x)&=\frac{2m}{x}J_m(x)-J_{m-1}(x)\\
        I_{m+1}(x)&=-\frac{2m}{x}I_m(x)+I_{m-1}(x)
    \end{split}
\end{equation*}
\end{proposition}
We will consider quotients of successive modified Bessel functions.
\begin{definition}
\label{def:besselq}
For $m\in\Qb$, let~$F_m$ be the following meromorphic function on~$\Cb\setminus(-\infty, 0]$.
$$F_m(x):=\frac{I_{m}(x)}{x I_{m-1}(x)}\ .$$
\end{definition}

The following identity, which can be viewed as a discrete Riccati equation will be important in the sequel.
\begin{keyid}
\label{lem:key}
For $m\in\Qb$,  we have
\begin{equation*}
    x^2 F_{m+1}(x)F_m(x)=1-2mF_m(x)
\end{equation*}
\end{keyid}
\begin{proof} From the definition of $F_m$ and Proposition~\ref{prop:besselrec} we have
    $$x^2F_{m+1}F_m=x^2\cdot\frac{I_{m+1}}{x I_m}\cdot\frac{I_m}{xI_{m-1}}=\frac{1}{I_{m-1}}\left(I_{m-1}-\frac{2m}{x}I_m\right)=1-2mF_m\ .$$
\end{proof}
\section{Second order recursion for cross products of Bessel functions}
Fix~$m\in\Qb$. The sequence~$(W_{m+n})_{n=0}^{\infty}$ satisfies a fourth order linear recurrence with non-constant coefficients in~$\Qb(x)$~(see~\cite{lvov-man}). However, it also satisfies a second order linear recurrence whose coefficients, while not even algebraic, satisfy themselves a quadratic recursion. 
We prove
\begin{theorem}
\label{thm:wrec}
    Let $m\in\Qb$. The following recursions formulae hold.
    \begin{equation*}
    \label{Wrec}
       W_{m+1}=2m F_{m} W_{m}-(1-2mF_{m})W_{m-1}
    \end{equation*}
    and
    \begin{equation*}
        W_{m-1}=2m G_m W_m-\bigl(1+2m G_m) W_{m+1}
    \end{equation*}
    where $G_m=1/(x^2F_{m+1})=I_{m}/(x I_{m+1})$.
\end{theorem}
\begin{proof}
On the one hand we have
by~\eqref{eqn:W-def} and Proposition~\ref{prop:besselrec}
\begin{equation*}
    \begin{split}
    W_{m-1}+W_{m+1}&= (I_m J_{m-1}+ I_{m-1} J_m)+(I_m J_{m+1}-I_{m+1}J_m)\\
    &=I_m  \frac{2m}{x}J_m + \frac{2m}{x}I_m J_m=\frac{4m}{x} I_m J_m\ .
    \end{split}
\end{equation*}

On the other hand, 
   \begin{equation*}
       \begin{split}
       2m F_m (W_m+W_{m-1}) = 2m \frac{I_m}{x I_{m-1}} (I_{m-1}J_m-I_m J_{m-1}+I_m J_{m-1}+I_{m-1} J_m)=\frac{4m}{x} I_m J_m\ .
       \end{split}
   \end{equation*}
   Comparing the preceding expressions gives the forward recursion formula.
   The backward recursion formula follows immediately from the forward one once we take into account Key Identity~\ref{lem:key}.
\end{proof}
\section{Rolling out the recursion}
\label{sec:roll-forward}
In this section we use the  recursion formula for~$W_m$ (Theorem~\ref{thm:wrec}) in order to express any element in the sequence in terms of two initial consecutive terms.
\begin{proposition}
\label{prop:rec-rolled-forward}
     Let $m\in\Qb$, and $n\in\Nbz$. There exist polynomials $A_{m,n}$, $B_{m,n}$,~$\tilde{B}_{m, n}$ and~$C_{m,n}\in \Qb[x]$ such that:
    \begin{equation}
  \begin{split}
        x^{2n}W_{m+n+1}=\bigl(A_{m,n}& F_{m}+x^2B_{m,n}+C_{m,n}F_{m}^{-1}\bigr)W_{m}+
        \\
        &\bigl(A_{m,n}F_{m}+\tilde{B}_{m,n}-C_{m,n}F_{m}^{-1}\bigr)W_{m-1}
    \end{split}
    \end{equation}
\end{proposition}
\begin{remark}
    Note that the coefficients in the preceding formula are of degree one in~$F_m$ and~$F_{m}^{-1}$.
\end{remark}
\begin{proof}
    We prove the claim by induction on $n$. The case $n=0$ follows from Theorem~\ref{thm:wrec}. For $n\geq 1$
    \begin{equation*}
    \begin{split}
        x^{2n}&W_{m+n+1}=x^2x^{2n-2}W_{(m+1)+(n-1)+1}=\\
        &\bigl(x^2A_{m+1,n-1}F_{m+1}+x^4B_{m+1,n-1}+x^2C_{m+1,n-1}F_{m+1}^{-1}\bigr) W_{m+1}\\
        +&\bigl(x^2A_{m+1,n-1}F_{m+1}+x^2\tilde{B}_{m+1,n-1}-x^2C_{m+1,n-1}F_{m+1}^{-1}\bigr)W_{m}
    \end{split}
    \end{equation*}
    We substitute $W_{m+1}$ using Theorem~\ref{thm:wrec} and Key Identity~\ref{lem:key}.
    \begin{equation*}
    \begin{split}
        x^{2n}&W_{m+n+1}=\\ 
        &\bigl(2mx^2A_{m+1,n-1} F_m F_{m+1} + 2mx^4B_{m+1,n-1}F_m + 2mx^2C_{m+1,n-1}F_m F_{m+1}^{-1}\bigr)W_{m}\\
        -&\bigl(x^2 A_{m+1,n-1}F_{m+1}+x^4B_{m+1,n-1}+x^2C_{m+1,n-1}F_{m+1}^{-1}\bigr)x^2F_{m}F_{m+1}W_{m-1}\\
        +&\bigl(x^2A_{m+1,n-1}F_{m+1}+x^2\tilde{B}_{m+1,n-1}-x^2C_{m+1,n-1}F_{m+1}^{-1}\bigr)W_{m}
    \end{split}
    \end{equation*}
    Applying Key Identity~\ref{lem:key} and collecting terms gives
      \begin{equation*}
    \begin{split}
        x^{2n}& W_{m+n+1}=\\ 
        \Bigl( &2m A_{m+1,n-1}(1-2mF_m)+2m x^4B_{m+1,n-1}F_m \\
         &- x^2C_{m+1,n-1}(1-2m F_m) F_{m+1}^{-1}+ A_{m+1,n-1}(F_m^{-1}-2m)+x^2\tilde{B}_{m+1,n-1}\Bigr)W_{m}
        \\ -\Bigl(&x^2A_{m+1,n-1}F_{m+1}(1-2mF_m) +x^4B_{m+1,n-1}(1-2mF_m)+x^4C_{m+1,n-1}F_{m}\Bigr)W_{m-1}
    \end{split}
    \end{equation*}
    
    Applying once more Key Identity~\ref{lem:key} and collecting terms gives
     \begin{equation*}
    \begin{split}
        x^{2n}&W_{m+n+1}=\\ 
        \Bigl(&\bigl(-4m^2A_{m+1, n-1}+2mx^4B_{m+1, n-1}-x^4C_{m+1, n-1}\bigr)F_m \\
        &+x^2\tilde{B}_{m+1, n-1}+A_{m+1,n-1}F_{m}^{-1}\Bigr)W_{m}
        \\ 
        +\Bigl(&2m A_{m+1, n-1}(1-2m F_m)-A_{m+1, n-1} ( F_{m}^{-1}-2m)+\left(2m x^4B_{m+1, n-1}-x^4C_{m+1, n-1}\right)F_m \\
        & -x^4B_{m+1, n-1} \Bigr)W_{m-1}
    \end{split}
    \end{equation*}
and finally, 
      \begin{equation*}
    \begin{split}
        x^{2n}&W_{m+n+1}=\\ 
        \Bigl(&\bigl(-4m^2A_{m+1, n-1}+2mx^4B_{m+1, n-1}-x^4C_{m+1, n-1}\bigr) F_m\\
        &+x^2\tilde{B}_{m+1, n-1}+ A_{m+1,n-1}F_{m}^{-1}\Bigr)W_{m}
        \\ 
        +\Bigl(&\bigl(-4m^2A_{m+1, n-1}+2m x^4B_{m+1, n-1}-x^4C_{m+1, n-1}\bigr)F_m\\
        & +4m A_{m+1, n-1}-x^4B_{m+1, n-1} -A_{m+1, n-1} F_{m}^{-1}\Bigr)W_{m-1}
    \end{split}
    \end{equation*}
    which is of the desired form.
    \end{proof}

    As an immediate consequence of the above computation, we obtain the following lemma.
    \begin{lemma}
    \label{lem:coeff-rec}
        Let $A_{m,n}$, $B_{m,n}$, $\tilde{B}_{m,n}$, and $C_{m,n}$ be as in Proposition~\ref{prop:rec-rolled-forward}. Then, the following recurrence relations hold.
        \begin{enumerate}[label=(\roman*)]
        \item \label{item:A} $A_{m, 0}=2m$; $A_{m,n}=-4m^2A_{m+1,n-1}+2mx^4B_{m+1,n-1}-x^4C_{m+1,n-1}$
             \item $B_{m, 0} = 0$; $B_{m,n}=\tilde{B}_{m+1,n-1}$
      \item $\tilde{B}_{m, 0}=-1$; $\tilde{B}_{m,n}=4m A_{m+1,n-1}-x^4B_{m+1,n-1}$
        \item $C_{m, 0}=0$; $C_{m,n}=A_{m+1,n-1}$
   
    \end{enumerate}
    \end{lemma}
As a corollary, we have
    \begin{lemma}
    \label{lem:coeff-mod}
    Let $m\in \Qb$ and $n\in\Nb$. Then,
$$
            A_{m,n}\equiv 2(-4)^{n}(m+n)\prod_{k=0}^{n-1}(m+k)^2\mod x^4\ .
$$
    \end{lemma}
    In particular, if $m\not\in \Zb$ or $m>0 $ or $m< -n$, then
    $A_{m,n}\not\equiv 0 \mod{x}$.
    \begin{proof}
    The proof follows immediately from Lemma~\ref{lem:coeff-rec}, part~\ref{item:A}.
    \end{proof}

The next proposition rolls the recursion backward and shows the connection to the forward recursion.
\begin{proposition}
\label{prop:rec-rolled-backward}
     Let $m\in\Qb$, and $n\in\Nbz$. With the same notations as in~Prop.~\ref{prop:rec-rolled-forward} 
    \begin{equation}
  \begin{split}
        (-1)^{n+1}x^{2n}W_{m-n-1}=
        &\bigl(A_{-m,n}G_{m}+x^2B_{-m,n}+C_{-m,n}G_{m}^{-1}\bigr)W_{m}+
        \\
        &-\bigl(A_{-m,n}G_{m}+\tilde{B}_{-m,n}-C_{-m,n}G_m^{-1}\bigr)W_{m+1}
    \end{split}
    \end{equation}
\end{proposition}
\begin{proof}[Sketch of Proof]
   In case $m\in\Zb$ the proposition follows immediately from Prop.~\ref{prop:rec-rolled-forward}, since $W_{-m}=(-1)^m W_{m}$ and $F_{-m}=G_m$.
 For any~$m\in\Qb$ one follows the same pattern of proof of Proposition~\ref{prop:rec-rolled-forward}. The Key Identity~\ref{lem:key} is replaced by the identity 
    \[ x^2 G_m G_{m-1}=1+2m G_m\ .\]
    We need also to use the recursion relations for the coefficients, given in~Lemma~\ref{lem:coeff-rec}.
\end{proof}

\section{Proof of Theorem~\ref{thm:mult3}}
    We recall
    \begin{proposition}[\cite{lvov-man}]
    \label{prop:m-m+1}
        Let $m\in\Qbp$ or $m\in\Zb$. The functions $W_m$ and $W_{m+1}$ have no joint positive roots.
    \end{proposition}
    \begin{proof}
       Assume $W_m(x_0)=W_{m+1}(x_0)=0$ for some $x_0>0$. Observe that 
       $$W_{m+1}+W_m=I_{m}J_{m+1}-I_{m+1}J_{m}+I_{m+1}J_m+I_mJ_{m+1}=2I_mJ_{m+1}$$
       and 
       $$W_{m+1}-W_m=I_{m}J_{m+1}-I_{m+1}J_{m}-I_{m+1}J_m-I_mJ_{m+1}=-2I_{m+1}J_{m}$$
       If $m\in\Zb$ or~$m\geq 0$, then $I_m(x_0)>0$ (see Lemma~\ref{lem:I-positive}). In view of this observation and  the above formulae, the assumption implies that $J_m(x_0)=J_{m+1}(x_0)=0$. However, this is impossible since it would imply that $J_m'(x_0)=(m/x_0)J_m(x_0)-J_{m+1}(x_0)$ is also zero, while $J_m$ satisfies a second order linear ODE.
    \end{proof}
    A direct consequence of the preceding proposition and Proposition \ref{prop:rec-rolled-forward} is
    \begin{corollary}
        \label{cor:x0-equations}
    Let~$n\in\Nb_0$ and $x_0>0$. 
    \begin{enumerate}[label=(\alph*)]
    \item \label{+n}
    If~$x_0$ is a joint root of $W_{m}$, and $W_{m+n+2}$ where $m\in \Qbp$, then
    \begin{equation*}
        \label{eqn:+n}
            A_{m+1, n}(x_0)F_{m+1}(x_0)^2+x_0^2 B_{m+1,n}(x_0)F_{m+1}(x_0)+C_{m+1,n}(x_0)=0
    \end{equation*}
    \item \label{-n}
    If~$x_0$ is a joint root of $W_{m}$, and $W_{m-n-2}$ where $m-1\in \Qbp$, then 
    \begin{equation*}
            \begin{split}
            A_{-m+1, n}(x_0) &\bigl(x_0^2 F_{m+1}(x_0)+2m\bigr)^2\\ &+x_0^4B_{-m+1, n}(x_0)\left(x_0^2F_{m+1}(x_0)+2m\right)+x_0^4C_{-m+1, n}(x_0)=0
            \end{split}
    \end{equation*}
    \end{enumerate}
    \end{corollary}
\begin{proof}
Part~\ref{+n} follows from Prop.~\ref{prop:rec-rolled-forward} with $m$ replaced by $m+1$, taking into account Prop.~\ref{prop:m-m+1}.
To prove part~\ref{-n} observe first that  from Prop.~\ref{prop:rec-rolled-backward} with $m$ replaced by~$m-1$ and Prop.~\ref{prop:m-m+1} we get 
\[A_{-m+1, n}(x_0)G_{m-1}(x_0) +x_0^2 B_{-m+1, n}(x_0)+C_{-m+1, n}(x_0)G_{m-1}^{-1}(x_0)=0\ .\]
Multiply this equation by~$x_0^4 G_{m-1}$, while noting that $x^2G_{m-1}=x^2 F_{m+1}+2m$.
\end{proof}

\begin{proof}[Proof of Theorem \ref{thm:mult3}]
 Assume~$x_0>0$ and~$0\leq m_1<m_2<m_3$  are such that $W_{m_1}(x_0)=W_{m_2}(x_0)=W_{m_3}(x_0)=0$. By Proposition~\ref{prop:m-m+1} we can write $m_1=m_2-l-2$, $m_2=m$ and $m_3=m_2+n+2$ with $l, m, n\in \Nb_0$. 
 By Corollary~\ref{cor:x0-equations} setting $x=x_0$ solves a system 
   \begin{equation} 
   \label{eqn:x-F-alg-dep}
    \begin{cases}
         A_{m+1,n}(x)F_{m+1}(x)^2+ x^2B_{m+1, n}(x)F_{m+1}(x)+ C_{m+1, n}(x)=0  \\
         A_{-m+1,l}(x)\bigl(x^2F_{m+1}(x)+2m\bigr)^2\!\!+\!x^4B_{-m+1, l}(x)(x^2 F_{m+1}(x)+2m)\!+x^4C_{-m+1, l}(x)=0
     \end{cases}
     \end{equation}
Eliminating $F_{m+1}^2$ from the preceding system we obtain that $x_0$ is a root of an equation of the form
   $$\bigl(4m A_{m+1,n}(x)A_{-m+1,l}(x)+x^4P_1(x)\bigr)x^2F_{m+1}(x)+ 4m^2A_{m+1, n}(x)A_{-m+1, l}(x)+x^4P_2(x)=0$$
    for some polynomials $P_1, P_2\in \Qb[x]$, depending on $l, m, n$.

    By Lemma~\ref{lem:coeff-mod} and the fact that $m> l+1$ the polynomial $4mA_{m+1, n}A_{-m+1, l}+x^4P_1$ is not zero.  Hence, if it vanishes at the point~$x_0$ we get that $x_0$ is algebraic. Otherwise, using the preceding equation to eliminate $F_{m+1}$ from the first equation in~\eqref{eqn:x-F-alg-dep} leads to an equation of the form 
    $$16m^4A_{m+1, n}(x_0)^3A_{-m+1, l}(x_0)^2 +x_0^4 P_3(x_0)=0$$
    with $P_3\in \Qb[x]$ depending on $l, m, n$.
 From Lemma~\ref{lem:coeff-mod} it follows that $x_0$ is also algebraic in this case.
 We have shown that~$x_0$ is algebraic. However, this contradicts Prop.~\ref{prop:Bessel-root-trans}.
\end{proof}

\begin{proposition}[{\cite[cor.\ 6.4]{lvov-man}}]
\label{prop:Bessel-root-trans}
Let $x_0 >0$ be algebraic, and let $m\in\Qbp$.  Then,  $W_m(x_0)\neq 0$.
\end{proposition}
\begin{proof}
In case $2m$ is not an odd integer,
it was proved by Siegel in~\cite{siegel} that  $J_m,J'_m,I_m,I'_m$ are algebraically independent functions  over $\mathbb{C}(x)$. Observe that
the following ODE is satisfied
$$\begin{pmatrix}
J_m\\
J'_m\\
I_m\\
I'_m
\end{pmatrix}'=\begin{pmatrix}
0 & 1 & 0 & 0\\
-1+\frac{m^2}{x^2} & -\frac{1}{x} & 0 & 0\\
0 & 0 & 0 & 1\\
0 & 0 & 1+\frac{m^2}{x^2} & -\frac{1}{x}
\end{pmatrix}\cdot\begin{pmatrix}
J_m\\
J'_m\\
I_m\\
I'_m
\end{pmatrix}$$
Hence, if $x_0$ is algebraic, it follows from the Siegel-Shidlovskii theory, that $J_m(x_0)$, $J'_m(x_0)$, $I_m(x_0)$, $I'_m(x_0)$ are algebraically independent and in particular $W_m(x_0)\neq0$.

    In case~$2m$ is odd the statement in the proposition is simpler, since it is essentially a special case of Lindemann-Weierstrass Theorem~\ref{LW}. Indeed,
    Suppose that $W_m(x_0)=0$. Then, $\frac{I_{m+1}}{I_m}(x_0)=-\frac{J_{m+1}}{J_m}(x_0)$. Hence, combining this with 
    Lemma~\ref{lem:tr-deg-ind} iterated we get
    \[\begin{split}
        1&\geq \tr. \deg\Big(\frac{I_{m+1}}{I_m}(x_0)\Big)=\tr. \deg\Big(x_0, \frac{I_{m+1}}{I_m}(x_0),\frac{J_{m+1}}{J_m}(x_0)\Big)
        \\& =\tr. \deg \Big(x_0, \frac{I_{1/2}}{I_{-1/2}}(x_0), \frac{J_{1/2}}{J_{-1/2}}(x_0)\Big)
        \stackrel{\text{Lem.~\ref{lem:JIonehalf}}}{=}\tr.\deg \big(\tanh(x_0), \tan(x_0)\big)\\ &=\tr.\deg (\me^{x_0}, \me^{\mi x_0})\stackrel{\text{L.-W.}}{=}2\ .
    \end{split}\]
    which is absurd.
\end{proof}
\begin{lemma} Let $m\in\Qbp$  and $x>0$. Then, 
    \label{lem:tr-deg-ind}
  \[  \tr. \deg \Big(x, \frac{I_{m+1}}{I_m}(x), \frac{J_{m+1}}{J_m}(x)\Big) = \tr. \deg \Big(x, \frac{I_{m}}{I_{m-1}}(x), \frac{J_{m}}{J_{m-1}}(x)\Big)\ .\]
  Here, if a  denominator vanishes we regard the corresponding quotient as algebraic.
\end{lemma}
\begin{proof}
    The lemma follows immediately from the identities
    \[\frac{I_{m+1}}{I_m}(x)=-\frac{2m}{x}+\frac{I_{m-1}}{I_m}(x),\quad \frac{J_{m+1}}{J_m}(x)=\frac{2m}{x}-\frac{J_{m-1}}{J_m}(x)\]
\end{proof}
\begin{theorem}[Lindemann-Weierstrass~\cites{lindemann, weierstrass}]
\label{LW}
Let $\alpha_1,\ldots,\alpha_n\in\Cb$ be algebraic numbers that are linearly independent over $\Qb$. Then, the numbers $e^{\alpha_1},\ldots,e^{\alpha_n}$ are algebraically independent.
\end{theorem}
\section{Schanuel's conjecture and the trivial multiplicities conjecture}
\label{sec:schanuel}
We observe that the classical Schanuel's conjecture from Transcendental Number Theory implies that in odd dimensions non-trivial multiplicities in the spectrum of the clamped round plate do not exist. 
This leads us to  a formulation of a Schanuel-type conjecture which would eliminate non-trivial multiplicities in even dimensions.
First, let us record the following 
special case of Schanuel's conjecture (see~\cite[ch. III, historical note, p.~30]{lang-book}
\begin{conjecture}[Schanuel's conjecture - special case]
\label{conj:schanuel}
Let $x\in\Rb$ be non-zero. Then 
\[\tr. \deg(x, \me^{x}, \me^{\mi x})\geq 2\ .\] 
\end{conjecture}

\begin{proposition}
\label{prop:conjectures-implied}
Conjecture~\ref{conj:schanuel} implies Conjecture~\ref{conj:trivial-mult} in odd dimensions.
\end{proposition}
\begin{proof} If $W_{m}(x_0)=W_{m+n}(x_0)=0$ for some $m\in\Qbp$ and $n\in \Nb_0$,  then the first equation in~\eqref{eqn:x-F-alg-dep} shows that~$x_0$ and $I_{m+1}(x_0)/I_{m}(x_0)$ are algebraically dependent. Moreover,
$\frac{I_{m+1}}{I_m}(x_0)+\frac{J_{m+1}}{J_m}(x_0)=0$. It follows that
\[ \tr. \deg\Big(x_0, \frac{I_{m+1}}{I_m}(x_0), \frac{J_{m+1}}{J_m}(x_0)\Big)\leq 1\ .\] 
From Lemma~\ref{lem:tr-deg-ind} we conclude that 
\[ \tr. \deg\Big(x_0, \frac{I_{1/2}}{I_{-1/2}}(x_0), \frac{J_{1/2}}{J_{-1/2}}(x_0)\Big)\leq 1\ .\]
Hence, by Lemma~\ref{lem:JIonehalf} we obtain that
\[ \tr. \deg\big(x, \tanh(x), \tan(x)\big)\leq 1\ .\]
On the other hand, Conjecutre~\ref{conj:schanuel} shows that
\[ \tr. \deg\big(x, \tanh(x), \tan(x)\big)\geq 2\ ,\]
and we get a contradiction.
\end{proof}

The argument in the preceding proof shows that the following conjecture implies the non-trivial multiplicity conjecture in even dimensions.
\begin{conjecture}[Schanuel-type conjecture]
\label{conj:schanuel-bessel}
Let $x>0$. Then
\[\tr. \deg\Big(x, \frac{I_1}{I_0}(x), \frac{I_1}{I_0}(\mi x)\Big)\geq 2\ .\]
\end{conjecture}
\begin{remark} Here, as in Lemma~\ref{lem:tr-deg-ind}, if $I_0(\mi x)=0$ we interpret the quotient $(I_1/I_0)(\mi x)$ as an algebraic number.
\end{remark}
\printbibliography
\vspace{3ex}

\textsc{Einstein Institute of Mathematics, Edmond J. Safra Campus,
The Hebrew University of Jerusalem, Jerusalem 9190401, Israel}

\emph{Email address: }\texttt{\textbf{dan.mangoubi@mail.huji.ac.il}}
\vspace{3ex}

\textsc{Einstein Institute of Mathematics, Edmond J. Safra Campus,
The Hebrew University of Jerusalem, Jerusalem 9190401, Israel}

\emph{Email address: }\texttt{\textbf{daniel.rosenblatt2@mail.huji.ac.il}}
\end{document}